\documentclass{birkau}

\usepackage{amsmath,amssymb}
\usepackage[hidelinks]{hyperref}
\usepackage{enumitem}
\usepackage{tikz}
\usepackage{tikz-cd}
\usepackage{xcolor}

\usetikzlibrary{
  arrows.meta,
  shapes,
  matrix,
  fit,
  positioning,
  cd,
  intersections,
  patterns,
  patterns.meta,
  backgrounds,
  calc
}

\numberwithin{equation}{section}

\theoremstyle{plain}
\newtheorem{theorem}{Theorem}[section]
\newtheorem{lemma}[theorem]{Lemma}
\newtheorem{proposition}[theorem]{Proposition}
\newtheorem{corollary}[theorem]{Corollary}

\theoremstyle{definition}
\newtheorem{definition}[theorem]{Definition}

\newtheorem{remark}[theorem]{Remark}

\DeclareMathOperator{\Aut}{Aut}
\DeclareMathOperator{\Gal}{Gal}
\DeclareMathOperator{\Sym}{Sym}
\DeclareMathOperator{\Fix}{Fix}
\DeclareMathOperator{\Stab}{Stab}

\makeatletter
\def\@firstoffive#1#2#3#4#5{#1}
\def\@newref#1{%
  \expandafter\@setref\csname r@#1\endcsname\@firstoffive{#1}%
}
\def\ref#1{{\rm\@newref{#1}}}
\makeatother

\begin{document}

\title[Finite-Orbit Actions and Exact Reconstruction]
{Finite-Orbit Actions\\ and Exact Reconstruction}

\author[N. Marmaridis]{Nikolaos Marmaridis}

\address{Mathematics Department\\University of Ioannina\\451 10 Ioannina\\Greece}
\urladdr{https://sites.google.com/site/maths4edu/}
\email{nmarmar@uoi.gr}

\subjclass{20E18, 20E22, 12F10, 12F05, 54H11}

\keywords{Profinite groups, Group actions, Finite-level exactness property,
Galois correspondence, Infinite Galois theory}

\begin{abstract}
We associate a profinite group to every group \(G\) acting on a set
\(\Omega\) with finite orbits. For each finite \(G\)-stable subset
\(A\subseteq\Omega\), let \(G_A\leq\Sym(A)\) be the induced finite
permutation group. The groups \(G_A\), with the natural restriction maps,
form an inverse system, and we define
$\Gamma_\Omega:=\varprojlim_A G_A$.

We show that \(\Gamma_\Omega\) acts naturally on \(\Omega\) and is
canonically topologically isomorphic to the closure of the image of \(G\) in
\(\Sym(\Omega)\), endowed with the topology of pointwise convergence.

We introduce the finite-level exactness property \(\textup{FLEP}\), under
which subgroups of \(\Gamma_\Omega\) are recovered up to closure from their
fixed-point sets, and closed subgroups are recovered exactly. We prove several
equivalent formulations of \(\textup{FLEP}\). Under this condition, the
fixed-point set construction gives an inclusion-reversing bijection between
closed subgroups of \(\Gamma_\Omega\) and the fixed subsets of \(\Omega\)
arising from closed subgroups.

We apply the theory in two directions. First, every profinite group \(\Gamma\)
is recovered from its normal finite-quotient action on
$\coprod_{N}\Gamma/N$, 
where \(N\) ranges over the open normal subgroups of \(\Gamma\). For this
action, \(\textup{FLEP}\) holds precisely when every finite quotient
\(\Gamma/N\), with \(N\) open and normal, is a Dedekind group. 

Second, if
\(G\leq\Aut(E)\) acts on a field \(E\) with finite orbits and \(F=E^G\), then
\(E/F\) is Galois and the construction yields a canonical topological
isomorphism
$\Gamma_E\cong_{\mathrm{top}}\Gal(E/F)$, 
where \(\Gal(E/F)\) has the Krull topology. Thus the Krull Galois group is
recovered from finite-orbit data.
\end{abstract}
\maketitle

\section{Introduction}

Let \(G\) be a group acting on a set \(\Omega\), and suppose that every orbit
\(G\omega\), with \(\omega\in\Omega\), is finite. For each finite
\(G\)-stable subset \(A\subseteq\Omega\), the restricted action gives a finite
permutation group \(G_A\leq\Sym(A)\). As \(A\) varies, these groups form an
inverse system, and we associate with the action the profinite group
\[
\Gamma_\Omega:=
\varprojlim_{A\in\mathcal F_G(\Omega)}G_A.
\]
The aim of this paper is to study this construction as a finite-level
reconstruction procedure and to determine when it admits a Galois-type
fixed-point correspondence.

Our first result identifies the inverse-limit group \(\Gamma_\Omega\) with the
closure of the permutation image of \(G\). If
\(\rho:G\to\Sym(\Omega)\) is the permutation representation and
\(\Sym(\Omega)\) is endowed with the topology of pointwise convergence, then
\[
\Gamma_\Omega\cong_{\mathrm{top}}\overline{\rho(G)}.
\]
Thus the inverse limit of the finite permutation groups \(G_A\) recovers
precisely the closure of the original permutation representation. This places the construction in the general circle of completion procedures
arising from permutation representations and homomorphisms with dense image
into totally disconnected locally compact groups; see \cite{ReidWesolek}. In
the present paper, the finite-orbit hypothesis leads specifically to profinite
completions.

The second part of the paper concerns exactness. The group
\(\Gamma_\Omega\) acts naturally on \(\Omega\), and one may ask when subgroups
of \(\Gamma_\Omega\) are recovered from their fixed-point sets. This leads to the
finite-level exactness property, abbreviated \(\textup{FLEP}\). In its open
subgroup form, \(\textup{FLEP}\) requires
\[
U=\Stab_{\Gamma_\Omega}(\Fix_\Omega(U))
\]
for every open subgroup \(U\leq\Gamma_\Omega\). We prove several equivalent
formulations, including the closed-subgroup form
\[
J=\Stab_{\Gamma_\Omega}(\Fix_\Omega(J)).
\]
Under \(\textup{FLEP}\), the assignments
\(J\mapsto\Fix_\Omega(J)\) and
\(X\mapsto\Stab_{\Gamma_\Omega}(X)\) give an inclusion-reversing bijection
between closed subgroups of \(\Gamma_\Omega\) and the fixed-point subsets arising
from them.

The theory is applied first to profinite groups. If \(\Gamma\) is profinite
and
\[
\Omega_\Gamma:=\coprod_{N\in\mathcal N_o(\Gamma)}\Gamma/N,
\]
where \(\mathcal N_o(\Gamma)\) denotes the set of open normal subgroups of
\(\Gamma\), then the left-translation action of \(\Gamma\) on
\(\Omega_\Gamma\) has finite orbits and reconstructs \(\Gamma\). For this
normal finite-quotient action, \(\textup{FLEP}\) holds if and only if every
open subgroup of \(\Gamma\) is normal, equivalently if and only if each finite
quotient \(\Gamma/N\), with \(N\in\mathcal N_o(\Gamma)\), is a Dedekind
group. Thus the normal finite-quotient action shows that reconstruction and
exactness are distinct phenomena: the action always reconstructs \(\Gamma\),
but it has an exact fixed-point set correspondence only under a strong
finite-quotient condition.

The final application is to infinite Galois theory. Let \(E\) be a field, let
\(G\leq\Aut(E)\), and put \(F=E^G\). 
The finite orbit property for the action
of \(G\) on \(E\) is equivalent to \(E/F\) being Galois. For each finite
\(G\)-stable subset \(A\subseteq E\), the field \(L_A:=F(A)\) is finite
Galois over \(F\), and \(G_A\) identifies canonically with
\(\Gal(L_A/F)\). Passing to inverse limits gives a canonical topological
isomorphism
\[
\Gamma_E\cong_{\mathrm{top}}\Gal(E/F),
\]
where the right-hand side has the Krull topology. Thus the Krull Galois
group, its topology, and the classical closed-subgroup Galois correspondence
are recovered from finite orbit data. In this case, finite Galois theory at
the levels \(L_A/F\) supplies the finite-level separation required for
\(\textup{FLEP}\).

A closely related finite-orbit condition appears in infinite Galois theory for
commutative rings. Cruthirds calls a group \(G\) of automorphisms of a
commutative ring \(S\) locally finite when every orbit
\(\{\sigma(s):\sigma\in G\}\) is finite, and studies fixed-ring
correspondences for such groups; see \cite{Cruthirds}.

The arguments use standard facts on inverse limits of finite groups and
profinite groups, as presented for example in
\cite{RibesZalesskii, Wilson}, the topology of pointwise convergence
on permutation groups, Artin's theorem on fixed fields of finite groups of
automorphisms, and the classical Krull theory of infinite Galois extensions;
for the latter, see for instance
\cite{ConradInfiniteGalois, LangAlgebra}.

While permutation completions are classical, the present paper develops a
systematic finite-orbit reconstruction framework, separates reconstruction
from exactness through the property \(\textup{FLEP}\), and uses this viewpoint
to give a unified treatment of profinite reconstruction, fixed-point set
correspondences, and infinite Galois theory.

The paper is organized as follows. Section~\ref{sec:fop} introduces finite-orbit actions
and their finite permutation levels. Section~\ref{sec:gamma} defines \(\Gamma_\Omega\) and
records its basic topological properties. Section~\ref{sec:gamma-action}  constructs the induced
action of \(\Gamma_\Omega\) on \(\Omega\). Section~\ref{sec:closure} identifies
\(\Gamma_\Omega\) with the permutation closure of \(G\). Section~\ref{sec:flep} introduces
\(\textup{FLEP}\) and proves the fixed-set correspondence. Section~\ref{sec:normal-finite-quotient-action} applies
the theory to profinite groups. Section~\ref{sec:galois} applies the construction to infinite
Galois theory.

\section{Finite-orbit actions and finite levels}
\label{sec:fop}

Let \(G\) be a group acting on a set \(\Omega\). We write the action on the
left, so that \(g\omega\) denotes the image of \(\omega\in\Omega\) under
\(g\in G\). Let \(\rho:G\to\Sym(\Omega)\), \(\rho(g)(\omega)=g\omega\), be
the associated permutation representation.

\begin{definition}
The action of \(G\) on \(\Omega\) is said to have the \emph{finite orbit
property}, abbreviated \textup{FOP}, if the orbit
\(G\omega:=\{g\omega:g\in G\}\) is finite for every \(\omega\in\Omega\).
\end{definition}

A subset \(A\subseteq\Omega\) is called \(G\)-stable if \(gA=A\) for every
\(g\in G\). Equivalently, \(ga\in A\) whenever \(a\in A\) and \(g\in G\).
Let \(\mathcal F_G(\Omega)\) denote the set of all finite \(G\)-stable subsets
of \(\Omega\). This set is directed by inclusion, since \(A\cup B\) is again
finite and \(G\)-stable whenever \(A,B\in\mathcal F_G(\Omega)\). If the action
has \textup{FOP}, then every element of \(\Omega\) lies in some member of
\(\mathcal F_G(\Omega)\), namely in its own orbit.

For \(A\in\mathcal F_G(\Omega)\), the restricted action induces a homomorphism
\[
\rho_A:G\to\Sym(A),
\qquad
\rho_A(g)(a)=ga.
\] We write
\(G_A:=\rho_A(G)\leq\Sym(A)\).

Suppose that \(A\subseteq B\) are finite \(G\)-stable subsets. Restriction of
permutations from \(B\) to \(A\) gives a map \(r_{BA}:G_B\to G_A\), defined by
\(r_{BA}(\rho_B(g))=\rho_A(g)\).

\begin{lemma}
The map \(r_{BA}\) is a well-defined surjective homomorphism.
\end{lemma}

\begin{proof}
If \(\rho_B(g)=\rho_B(h)\), then \(g\) and \(h\) induce the same permutation
on \(B\), hence also on \(A\). Thus \(\rho_A(g)=\rho_A(h)\), so \(r_{BA}\) is
well-defined. It is a homomorphism because
\(r_{BA}(\rho_B(g)\rho_B(h))=r_{BA}(\rho_B(gh))=\rho_A(gh)=\rho_A(g)\rho_A(h)\).
It is surjective because every element of \(G_A\) has the form \(\rho_A(g)\)
for some \(g\in G\), and \(r_{BA}(\rho_B(g))=\rho_A(g)\).
\end{proof}

\begin{lemma}
\label{lem:inversesystem}
The family \((G_A,r_{BA})_{A\subseteq B}\) is an inverse system of finite
groups over the directed set \(\mathcal F_G(\Omega)\).
\end{lemma}

\begin{proof}
For every \(A\), the map \(r_{AA}\) is the identity on \(G_A\). If
\(A\subseteq B\subseteq C\), then restriction from \(C\) to \(A\) is the same
as restriction from \(C\) to \(B\), followed by restriction from \(B\) to \(A\).
Thus \(r_{CA}=r_{BA}\circ r_{CB}\), and the assertion follows.
\end{proof}

\section{The profinite group associated to a finite-orbit action}
\label{sec:gamma}

\begin{definition}
Let \(G\) act on \(\Omega\) with \textup{FOP}. We define
\[
        \Gamma_\Omega:=\varprojlim_{A\in\mathcal F_G(\Omega)}G_A,
\]
the inverse limit of the finite permutation groups \(G_A\) with respect to the
restriction maps \(r_{BA}:G_B\to G_A\), for \(A\subseteq B\). We endow
\(\Gamma_\Omega\) with the inverse-limit topology.
\end{definition}

Thus an element of \(\Gamma_\Omega\) is a compatible family
\(\gamma=(\gamma_A)_A\), with \(\gamma_A\in G_A\), such that
\(r_{BA}(\gamma_B)=\gamma_A\) whenever \(A\subseteq B\).

\begin{proposition}
The group \(\Gamma_\Omega\) is profinite.
\end{proposition}

\begin{proof}
Any inverse limit of finite discrete groups is a profinite group.
\end{proof}

\subsection*{Open subgroups and closures}
\label{subsec:open-subgroups-closures}

We record the standard descriptions of open subgroups and closures in the
present notation. For \(A\in\mathcal F_G(\Omega)\), let
\(\pi_A:\Gamma_\Omega\to G_A\) be the canonical projection,
\(\pi_A((\gamma_B)_B)=\gamma_A\). The subgroups \(\ker(\pi_A)\), with
\(A\in\mathcal F_G(\Omega)\), form a basis of open normal neighbourhoods of
the identity in \(\Gamma_\Omega\).

\begin{lemma}[Open subgroups]
\label{lem:open-subgroups-general}
A subgroup \(U\leq\Gamma_\Omega\) is open if and only if there exist
\(A\in\mathcal F_G(\Omega)\) and a subgroup \(Q\leq G_A\) such that
\(U=\pi_A^{-1}(Q)\). Equivalently, \(U\) is open if and only if
\(\ker(\pi_A)\leq U\) for some \(A\in\mathcal F_G(\Omega)\).
\end{lemma}

\begin{proof}
Since the kernels \(\ker(\pi_A)\) form a neighbourhood basis at the identity,
a subgroup \(U\leq\Gamma_\Omega\) is open precisely when it contains some
\(\ker(\pi_A)\). In that case \(U=\pi_A^{-1}(\pi_A(U))\), with
\(\pi_A(U)\leq G_A\). Conversely, every subgroup of the form
\(\pi_A^{-1}(Q)\), with \(Q\leq G_A\), contains \(\ker(\pi_A)\), and is
therefore open.
\end{proof}

\begin{lemma}[Closure formula]
\label{lem:closure-general}
For every subgroup \(H\leq\Gamma_\Omega\), its closure is given by
\[
        \overline H
        =
        \bigcap_{A\in\mathcal F_G(\Omega)}H\ker(\pi_A).
\]
\end{lemma}

\begin{proof}
In any topological group, \(\overline H=\bigcap_V HV\), where \(V\) ranges
over the neighbourhoods of the identity. Since the subgroups
\(\ker(\pi_A)\), with \(A\in\mathcal F_G(\Omega)\), form a neighbourhood
basis of the identity in \(\Gamma_\Omega\), this gives the formula.
\end{proof}
Thus, a subgroup \(H\leq\Gamma_\Omega\) is closed precisely when
\[
H=\bigcap_{A\in\mathcal F_G(\Omega)}H\ker(\pi_A).
\]
\begin{remark}
The closure formula also has the following coordinate interpretation:
\(\gamma\in\overline H\) if and only if, for every
\(A\in\mathcal F_G(\Omega)\), there exists an element
\(h^{(A)}\in H\), possibly depending on \(A\), such that
$\pi_A(\gamma)=\pi_A(h^{(A)})$.
\end{remark}

\section{The induced action of \(\Gamma_\Omega\) on
\(\Omega\)}
\label{sec:gamma-action}

The original action of \(G\) on \(\Omega\) extends naturally to an action of
\(\Gamma_\Omega\) on \(\Omega\).

Let \(\gamma=(\gamma_A)_A\in\Gamma_\Omega\) and let
\(\omega\in\Omega\). Choose \(A\in\mathcal F_G(\Omega)\) with 
\(\omega\in A \), and define
$\gamma\omega:=\gamma_A(\omega)$.

\begin{lemma}
The definition of \(\gamma\omega\) is independent of the choice of
\(A\in\mathcal F_G(\Omega)\) containing \(\omega\).
\end{lemma}

\begin{proof}
Let \(A,B\in\mathcal F_G(\Omega)\) with \(\omega\in A\cap B\), and set
\(C:=A\cup B\). Then \(C\in\mathcal F_G(\Omega)\). By compatibility,
\(r_{CA}(\gamma_C)=\gamma_A\) and \(r_{CB}(\gamma_C)=\gamma_B\). Hence the
actions of \(\gamma_A\) on \(A\) and of \(\gamma_B\) on \(B\) both agree with
the action of \(\gamma_C\) on \(C\). Since \(\omega\in A\cap B\), we get
\(\gamma_A(\omega)=\gamma_C(\omega)=\gamma_B(\omega)\).
\end{proof}

\begin{proposition}\label{prop:gamma-action}
The formula \((\gamma,\omega)\mapsto\gamma\omega\) defines a left action of
\(\Gamma_\Omega\) on \(\Omega\).
\end{proposition}

\begin{proof}
Let \(\gamma,\delta\in\Gamma_\Omega\) and let \(\omega\in\Omega\). Choose
\(A\in\mathcal F_G(\Omega)\) with \(\omega\in A\). Since
\(\delta_A\in G_A\leq\Sym(A)\), we have \(\delta\omega=\delta_A(\omega)\in A\).
Hence both \(\delta\omega\) and \(\gamma(\delta\omega)\) are computed using the
same finite \(G\)-stable set \(A\). It follows that
\(\gamma(\delta\omega)=\gamma_A(\delta_A(\omega))
=(\gamma_A\delta_A)(\omega)=(\gamma\delta)\omega\).
Similarly, the identity element acts trivially. Therefore the formula defines a
left action of \(\Gamma_\Omega\) on \(\Omega\).
\end{proof}

The original action is recovered through the canonical homomorphism
\[
\eta:G\to\Gamma_\Omega,
\qquad
\eta(g)=(\rho_A(g))_A .
\]
This homomorphism is well defined by compatibility of the restriction maps, and
\(\eta(g)\omega=g\omega\) for all \(g\in G\) and \(\omega\in\Omega\).

\begin{corollary}
\label{cor:orbits-coincide}
For every \(\omega\in\Omega\), the equality
\(\Gamma_\Omega\omega=G\omega\) holds. In particular, the induced action of
\(\Gamma_\Omega\) on \(\Omega\) has \textup{FOP}.
\end{corollary}

\begin{proof}
Let \(\omega\in\Omega\) and set \(A:=G\omega\). By \textup{FOP},
\(A\in\mathcal F_G(\Omega)\). If
\(\gamma=(\gamma_B)_B\in\Gamma_\Omega\), then
\(\gamma\omega=\gamma_A(\omega)\). Since \(\gamma_A\in G_A=\rho_A(G)\), there
exists \(g\in G\) with \(\gamma_A=\rho_A(g)\), and therefore
$\gamma\omega=g\omega\in G\omega$.

Thus \(\Gamma_\Omega\omega\subseteq G\omega\). The reverse inclusion follows
from \(\eta(g)\omega=g\omega\) for \(g\in G\). Hence
\(\Gamma_\Omega\omega=G\omega\). Since \(G\omega\) is finite, the induced
action of \(\Gamma_\Omega\) has \textup{FOP}.
\end{proof}

\begin{lemma}
The homomorphism \(\eta:G\to\Gamma_\Omega\), \(g\mapsto(\rho_A(g))_A\), has
dense image.
\end{lemma}
\begin{proof}
Let \(\gamma=(\gamma_A)_A\in\Gamma_\Omega\), and let
\(U(A,\gamma_A)=\{\delta\in\Gamma_\Omega:\delta_A=\gamma_A\}\) be a basic
open neighbourhood of \(\gamma\). Since \(\gamma_A\in G_A=\rho_A(G)\), there
exists \(g\in G\) such that \(\rho_A(g)=\gamma_A\). Hence
\(\eta(g)\in U(A,\gamma_A)\). Therefore every basic open neighbourhood of
every element of \(\Gamma_\Omega\) meets \(\eta(G)\), and so \(\eta(G)\) is
dense.
\end{proof}

\section{Realization as a permutation closure}
\label{sec:closure}

We now identify \(\Gamma_\Omega\) with the closure of the original permutation
group inside \(\Sym(\Omega)\). We endow \(\Sym(\Omega)\) with the topology of
pointwise convergence, equivalently the function topology. A basis of
neighbourhoods of the identity consists of the pointwise stabilizers of finite
subsets of \(\Omega\).

Let \(\rho:G\to\Sym(\Omega)\) be the permutation representation introduced in
Section~\ref{sec:fop}, and let \(\overline{\rho(G)}\) denote its closure in
\(\Sym(\Omega)\). Explicitly, \(\sigma\in\Sym(\Omega)\) belongs to
\(\overline{\rho(G)}\) if and only if, for every finite subset
\(Y\subseteq\Omega\), there exists \(g\in G\) such that
\(\sigma|_Y=\rho(g)|_Y\).

\begin{theorem}[Realization as a permutation closure]
\label{thm:permutation-closure}
Let \(G\) act on \(\Omega\) with \textup{FOP}. Then there is a natural
topological isomorphism
\[
        \Gamma_\Omega
        \cong_{\mathrm{top}}
        \overline{\rho(G)}
        \subseteq \Sym(\Omega).
\]
\end{theorem}
\begin{proof}
By Proposition~\ref{prop:gamma-action}, \(\Gamma_\Omega\) acts on \(\Omega\).
Let
\[
\Theta:\Gamma_\Omega\to\Sym(\Omega),
\qquad
\Theta(\gamma)(\omega)=\gamma\omega,
\]
be the associated permutation representation. Thus \(\Theta\) is a group
homomorphism.

We first show that \(\Theta\) is injective. Suppose
\(\Theta(\gamma)=\Theta(\delta)\). If \(A\in\mathcal F_G(\Omega)\), then for
every \(\omega\in A\),
\[
\gamma_A(\omega)
=
\Theta(\gamma)(\omega)
=
\Theta(\delta)(\omega)
=
\delta_A(\omega).
\]
Hence \(\gamma_A=\delta_A\) as permutations of \(A\). Since this holds for all
\(A\), we have \(\gamma=\delta\).

Next we show that \(\Theta\) is continuous. It is enough to check continuity
at the identity. Let \(Y\subseteq\Omega\) be finite, and let
\(W_Y\leq\Sym(\Omega)\) be the pointwise stabilizer of \(Y\). Choose
\(A\in\mathcal F_G(\Omega)\) with \(Y\subseteq A\). If
\(\gamma\in\ker(\pi_A)\), then \(\gamma_A=\mathrm{id}_A\), so
\(\Theta(\gamma)\) fixes every element of \(Y\). Thus
$
\Theta(\ker(\pi_A))\subseteq W_Y$.

Hence \(\Theta\) is continuous at the identity, and therefore continuous.

It remains to identify the image of \(\Theta\). We prove that
\[
\Theta(\Gamma_\Omega)=\overline{\rho(G)}.
\]
Since \(\Theta(\eta(g))=\rho(g)\) for every \(g\in G\), the permutation
image of \(G\) is contained in \(\Theta(\Gamma_\Omega)\). The group
\(\Gamma_\Omega\) is compact and \(\Theta\) is continuous, so
\(\Theta(\Gamma_\Omega)\) is compact. Since \(\Sym(\Omega)\) is Hausdorff,
\(\Theta(\Gamma_\Omega)\) is closed. Therefore
$\overline{\rho(G)}\subseteq \Theta(\Gamma_\Omega)$.

Conversely, let \(\gamma\in\Gamma_\Omega\). Let \(Y\subseteq\Omega\) be finite,
and choose \(A\in\mathcal F_G(\Omega)\) with \(Y\subseteq A\). Since
\(\gamma_A\in G_A=\rho_A(G)\), there exists \(g\in G\) such that
\(\gamma_A=\rho_A(g)\). Hence
$\Theta(\gamma)|_Y=\rho(g)|_Y$.
 Thus \(\Theta(\gamma)\in\overline{\rho(G)}\), and so
$\Theta(\Gamma_\Omega)\subseteq\overline{\rho(G)}$.

Therefore 
$\Theta(\Gamma_\Omega)=\overline{\rho(G)}$.

Thus \(\Theta\) is a continuous bijection from the compact group
\(\Gamma_\Omega\) onto the Hausdorff group \(\overline{\rho(G)}\). By the
compact-Hausdorff criterion, \(\Theta\) is a homeomorphism. Hence it is a
topological group isomorphism.
\end{proof}
In particular, after identifying \(\Gamma_\Omega\) with
\(\overline{\rho(G)}\leq\Sym(\Omega)\) by means of \(\Theta\), the
inverse-limit topology on \(\Gamma_\Omega\) is precisely the subspace topology
inherited from the topology of pointwise convergence on \(\Sym(\Omega)\).
\medskip 

By Corollary~\ref{cor:orbits-coincide}, the induced action of
\(\Gamma_\Omega\) on \(\Omega\) again has \textup{FOP}. Hence the finite-orbit
construction can be applied once more to this action.

\begin{corollary}[Idempotence of the construction]
\label{cor:idempotence}
The profinite group associated with the induced action of \(\Gamma_\Omega\) on
\(\Omega\) is naturally topologically isomorphic to~\(\Gamma_\Omega\).
\end{corollary}

\begin{proof}
By Theorem~\ref{thm:permutation-closure}, the profinite group obtained from
the induced action of \(\Gamma_\Omega\) on \(\Omega\) is topologically
isomorphic to the closure of the permutation image of \(\Gamma_\Omega\) in
\(\Sym(\Omega)\). This image is closed, since it is the continuous image of
the compact group \(\Gamma_\Omega\) in the Hausdorff group \(\Sym(\Omega)\).
The action is faithful, because an element fixing all points of \(\Omega\) has
all its coordinates equal to the identity. Hence the closed permutation image
is topologically isomorphic to~\(\Gamma_\Omega\).
\end{proof}

\begin{remark}
The preceding corollary is a special case of the following general fact. If a
profinite group \(\Gamma\) acts faithfully and continuously on a discrete set
\(\Omega\), then all orbits are finite, since they are compact subsets of a
discrete space. Hence the finite-orbit construction applies, and
Theorem~\ref{thm:permutation-closure} identifies the associated profinite
group with the closure of the permutation image of \(\Gamma\) in
\(\Sym(\Omega)\). This image is closed because \(\Gamma\) is compact and
\(\Sym(\Omega)\) is Hausdorff. By faithfulness, it is topologically
isomorphic to \(\Gamma\).
\end{remark}

\begin{remark}
The realization of \(\Gamma_\Omega\) as \(\overline{\rho(G)}\) is related to
completion constructions arising from homomorphisms with dense image into
totally disconnected locally compact groups; see \cite{ReidWesolek}. In the present paper, the finite-orbit hypothesis leads
specifically to profinite completions. The main focus is the exactness property
\textup{FLEP} and the resulting fixed-point correspondence.
\end{remark}

\section{Finite-level exactness and the fixed-set correspondence}
\label{sec:flep}

Throughout this section, let \(G\) act on \(\Omega\) with \textup{FOP}, and let
\(\Gamma_\Omega\) act on \(\Omega\) as in Section~\ref{sec:gamma-action}.

For \(H\leq\Gamma_\Omega\), let
$
\Fix_\Omega(H):=\{\omega\in\Omega:h\omega=\omega
\text{ for every } h\in H\}$ denote the fixed-point set of \(H\) in \(\Omega\).

For \(X\subseteq\Omega\), let 
$\Stab_{\Gamma_\Omega}(X):=\{\gamma\in\Gamma_\Omega:\gamma x=x
\text{ for every } x\in X\}$
denote the pointwise stabilizer of \(X\) in \(\Gamma_\Omega\).

\begin{lemma}
\label{lem:stabilizers-are-closed}
For every subset \(X\subseteq\Omega\), the pointwise stabilizer
\(\Stab_{\Gamma_\Omega}(X)\) is a closed subgroup of \(\Gamma_\Omega\).
\end{lemma}

\begin{proof}
Let \(H:=\Stab_{\Gamma_\Omega}(X)\). Clearly \(H\) is a subgroup of
\(\Gamma_\Omega\). We prove that \(\overline H\leq H\). Let
\(\gamma\in\overline H\) and \(x\in X\). Choose \(A\in\mathcal F_G(\Omega)\)
with \(x\in A\). By the coordinate interpretation following
Lemma~\ref{lem:closure-general}, there exists \(h^{(A)}\in H\) such that
\(\pi_A(\gamma)=\pi_A(h^{(A)})\). Since \(x\in A\), we have
$\gamma x=h^{(A)}x=x$.

Thus \(\gamma\) fixes every element of \(X\), so \(\gamma\in H\). Hence \(H\)
is closed.
\end{proof}

\begin{corollary}
\label{cor:fixed-points-closure}
For every subgroup \(H\leq\Gamma_\Omega\),
$\Fix_\Omega(H)=\Fix_\Omega(\overline H)$.
\end{corollary}

\begin{proof}
The inclusion \(\Fix_\Omega(\overline H)\subseteq\Fix_\Omega(H)\) is immediate
from \(H\leq\overline H\). Conversely, let
\(\omega\in\Fix_\Omega(H)\). Then
$H\leq\Stab_{\Gamma_\Omega}(\{\omega\})$.

By Lemma~\ref{lem:stabilizers-are-closed}, this stabilizer is closed, hence it
contains \(\overline H\). Therefore
\(\omega\in\Fix_\Omega(\overline H)\).
\end{proof}

The preceding corollary shows that fixed-point sets cannot distinguish a
subgroup from its closure. Thus the fixed-point set construction cannot give an
exact correspondence for arbitrary subgroups. The best possible statement is
that arbitrary subgroups are recovered up to closure, while closed subgroups
are recovered exactly. This is the situation familiar from infinite Galois theory, where a subgroup
of the Galois group is determined by its fixed field only after passing to its
topological closure.

We now isolate the finite-level condition which ensures that the same
phenomenon holds for the profinite group \(\Gamma_\Omega\) acting on
\(\Omega\).

\begin{definition}[Finite-level exactness]
\label{def:FLEP}
The induced action of \(\Gamma_\Omega\) on \(\Omega\) is said to satisfy the
\emph{finite-level exactness property}, abbreviated \(\textup{FLEP}\), if
\[
U=\Stab_{\Gamma_\Omega}(\Fix_\Omega(U))
\]
for every open subgroup \(U\leq\Gamma_\Omega\).
\end{definition}
\begin{proposition}[Equivalent forms of \textup{FLEP}]
\label{prop:FLEP-equivalent-forms}
For the action of \(\Gamma_\Omega\) on \(\Omega\), the following conditions
are equivalent.

\begin{enumerate}
\item[\textup{(F1)}]
For every open subgroup \(U\leq\Gamma_\Omega\),
\(U=\Stab_{\Gamma_\Omega}(\Fix_\Omega(U))\).

\item[\textup{(F2)}]
For every open subgroup \(U\leq\Gamma_\Omega\) and every
\(\gamma\in\Gamma_\Omega\setminus U\), there exists
\(\omega\in\Fix_\Omega(U)\) such that \(\gamma\omega\neq\omega\).

\item[\textup{(F3)}]
For every subgroup \(H\leq\Gamma_\Omega\),
\(\overline H=\Stab_{\Gamma_\Omega}(\Fix_\Omega(H))\).

\item[\textup{(F4)}]
For every closed subgroup \(J\leq\Gamma_\Omega\),
\(J=\Stab_{\Gamma_\Omega}(\Fix_\Omega(J))\).
\end{enumerate}
\end{proposition}

\begin{proof}
\(\textup{(F1)} \Leftrightarrow \textup{(F2)}\).
The inclusion
\(H\leq \Stab_{\Gamma_\Omega}(\Fix_\Omega(H))\)
holds for every subgroup \(H\leq\Gamma_\Omega\). Hence, for an open subgroup
\(U\leq\Gamma_\Omega\), the equality
\(U=\Stab_{\Gamma_\Omega}(\Fix_\Omega(U))\) fails precisely when there exists
\(\gamma\in\Gamma_\Omega\setminus U\) that fixes every element of
\(\Fix_\Omega(U)\). Equivalently, equality holds precisely when, for every
\(\gamma\in\Gamma_\Omega\setminus U\), there exists
\(\omega\in\Fix_\Omega(U)\) such that \(\gamma\omega\neq\omega\). This is
exactly \textup{(F2)}.

\(\textup{(F2)} \Rightarrow \textup{(F3)}\).
Let \(H\leq\Gamma_\Omega\). Since \(\Stab_{\Gamma_\Omega}(\Fix_\Omega(H))\)
is closed and contains \(H\), it contains \(\overline H\). It remains to prove
the reverse inclusion.

Let \(\gamma\notin\overline H\). By Lemma~\ref{lem:closure-general}, there
exists \(A\in\mathcal F_G(\Omega)\) such that
$\gamma\notin H\ker(\pi_A)$.
Set \(U:=H\ker(\pi_A)\). Since
\(\ker(\pi_A)\) is normal, \(U\) is a subgroup of \(\Gamma_\Omega\). Moreover,
\(\ker(\pi_A)\leq U\), hence \(U\) is open by
Lemma~\ref{lem:open-subgroups-general}. Also \(H\leq U\), while
\(\gamma\notin U\).

By \textup{(F2)}, there exists \(\omega\in\Fix_\Omega(U)\) such that
\(\gamma\omega\neq\omega\). Since \(H\leq U\), we have
\(\Fix_\Omega(U)\subseteq\Fix_\Omega(H)\). Therefore \(\gamma\) does not fix
all points of \(\Fix_\Omega(H)\), and hence
\(\gamma\notin\Stab_{\Gamma_\Omega}(\Fix_\Omega(H))\). This proves
\(\Stab_{\Gamma_\Omega}(\Fix_\Omega(H))\leq\overline H\). Thus
\textup{(F3)} holds.

\(\textup{(F3)} \Rightarrow \textup{(F4)}\) is immediate by taking \(H=J\)
closed.

\(\textup{(F4)} \Rightarrow \textup{(F1)}\). Indeed, every open subgroup of a
profinite group is closed, so \textup{(F4)} applies in particular to open
subgroups.

Therefore the four conditions are equivalent.
\end{proof}

\begin{remark}
Condition \(\textup{(F3)}\) is the form most closely analogous to infinite
Galois theory: a subgroup is determined up to closure by its fixed-point set,
and closed subgroups are determined uniquely by their fixed-point sets.
\end{remark}

\subsection*{The fixed-point set correspondence}

For the remainder of this section, assume that the action of
\(\Gamma_\Omega\) on \(\Omega\) satisfies \textup{FLEP}.

Let \(\mathcal J(\Gamma_\Omega)\) be the lattice of closed subgroups of
\(\Gamma_\Omega\), ordered by inclusion, and let
\(\mathcal X(\Omega):=\{\Fix_\Omega(J):J\in\mathcal J(\Gamma_\Omega)\}\)
be the family of fixed-point sets arising from closed subgroups, again ordered
by inclusion.

\begin{theorem}[Fixed-point set correspondence]
\label{thm:fixed-set-correspondence}
The assignments \(J\mapsto\Fix_\Omega(J)\) and
\(X\mapsto\Stab_{\Gamma_\Omega}(X)\) define mutually inverse
inclusion-reversing bijections between
\(\mathcal J(\Gamma_\Omega)\) and \(\mathcal X(\Omega)\).
\end{theorem}

\begin{proof}
Let \(J\in\mathcal J(\Gamma_\Omega)\). By \(\textup{FLEP}\), in the form
\(\textup{(F4)}\), we have
\[J=\Stab_{\Gamma_\Omega}(\Fix_\Omega(J)).\] 
Hence the composite
\(J\mapsto\Fix_\Omega(J)\mapsto
\Stab_{\Gamma_\Omega}(\Fix_\Omega(J))\) is the identity.

Conversely, let \(X\in\mathcal X(\Omega)\). Then \(X=\Fix_\Omega(J)\) for
some closed subgroup \(J\). The first part gives
\(\Stab_{\Gamma_\Omega}(X)=J\), and hence
\(\Fix_\Omega(\Stab_{\Gamma_\Omega}(X))=\Fix_\Omega(J)=X\). Thus the two
assignments are mutually inverse.

They reverse inclusions because \(J_1\leq J_2\) implies
\(\Fix_\Omega(J_2)\subseteq\Fix_\Omega(J_1)\), while \(X_1\subseteq X_2\)
implies
\(\Stab_{\Gamma_\Omega}(X_2)\leq\Stab_{\Gamma_\Omega}(X_1)\).
\end{proof}

\subsection*{Meets and joins}

We record how the fixed-point set correspondence behaves with respect to
intersections and closed generated subgroups. 

For closed subgroups
\(J,J^\prime\leq\Gamma_\Omega\), write \(\overline{\langle J,J^\prime\rangle}\) for the
closed subgroup generated by \(J\) and \(J^\prime\).

\begin{proposition}
\label{prop:lattice-formulas}
Assume \textup{FLEP}. Let \(J,J^\prime\leq\Gamma_\Omega\) be closed subgroups. Then
\[
\Fix_\Omega(\overline{\langle J,J^\prime\rangle})
=
\Fix_\Omega(J)\cap\Fix_\Omega(J^\prime).
\]
Moreover,
\[
\Fix_\Omega(J\cap J^\prime)
=
\Fix_\Omega\!\left(
\Stab_{\Gamma_\Omega}(\Fix_\Omega(J)\cup\Fix_\Omega(J^\prime))
\right).
\]
\end{proposition}

\begin{proof}
A point \(\omega\in\Omega\) is fixed by
\(\overline{\langle J,J^\prime\rangle}\) if and only if it is fixed by
\(\langle J,J^\prime\rangle\), by
Corollary~\ref{cor:fixed-points-closure}. This holds if and only if
\(\omega\) is fixed by every element of \(J\) and every element of
\(J^\prime\). This gives the first formula.

For the second formula, \(\textup{FLEP}\) gives
\[
\Stab_{\Gamma_\Omega}(\Fix_\Omega(J))=J
\quad\text{and}\quad
\Stab_{\Gamma_\Omega}(\Fix_\Omega(J^\prime))=J^\prime.
\]
Hence the pointwise stabilizer of
\(\Fix_\Omega(J)\cup\Fix_\Omega(J^\prime)\) is \(J\cap J^\prime\). Therefore
\[
\Fix_\Omega(J\cap J^\prime)
=
\Fix_\Omega\!\left(
\Stab_{\Gamma_\Omega}(\Fix_\Omega(J)\cup\Fix_\Omega(J^\prime))
\right),
\]
as required.\end{proof}

\begin{remark}
Under the fixed-point correspondence, closed subgroup joins correspond to
intersections of fixed-point sets. Closed subgroup meets correspond to the
closure operation naturally induced on the family \(\mathcal X(\Omega)\).
\end{remark}

\section{The normal finite-quotient action}
\label{sec:normal-finite-quotient-action}

We now specialize the construction to a profinite group. The goal is twofold:
to recover the group from a canonical finite-orbit action, and to characterize
when this action satisfies \textup{FLEP}.

Let \(\Gamma\) be a profinite group, and let
\(\mathcal N_o(\Gamma):=\{N\trianglelefteq_o\Gamma\}\) be the family of open
normal subgroups of \(\Gamma\). Consider the coproduct
\[
\Omega_\Gamma:=\coprod_{N\in\mathcal N_o(\Gamma)}\Gamma/N,
\]
where \(\Gamma/N=\{\delta N:\delta\in\Gamma\}\) is the set of left cosets of
\(N\). Since every open subgroup of a profinite group has finite index, each
\(\Gamma/N\) is finite.

The group \(\Gamma\) acts on \(\Omega_\Gamma\) by left translation:
if \(\omega=\delta N\in\Gamma/N\) and \(\gamma\in\Gamma\), set
\(\gamma\omega:=\gamma\delta N\). This is independent of the choice of
\(\delta\), and defines a left action of \(\Gamma\) on
\(\Omega_\Gamma\). Its orbits are precisely the components \(\Gamma/N\), so
the action has \textup{FOP}.

A finite \(\Gamma\)-stable subset of \(\Omega_\Gamma\) is necessarily a finite
union of components. Hence every \(A\in\mathcal F_\Gamma(\Omega_\Gamma)\) has
the form
\[
A=\Gamma/N_1\coprod\cdots\coprod\Gamma/N_r
\]
for some \(N_1,\ldots,N_r\in\mathcal N_o(\Gamma)\). 

\noindent For such an \(A\), let
\(\rho_A:\Gamma\to\Sym(A)\), \(\rho_A(\gamma)(\omega)=\gamma\omega\), be the
restricted permutation representation, and let
\(\Gamma_A:=\rho_A(\Gamma)\).

The groups \(\Gamma_A\), with the
restriction maps \(r_{BA}:\Gamma_B\to\Gamma_A\), form an inverse system over
\(\mathcal F_\Gamma(\Omega_\Gamma)\); see~Lemma~\ref{lem:inversesystem}. 

\noindent Since \(A\subseteq B\) implies
\(\ker\rho_B\leq\ker\rho_A\), the finite quotients
\(\Gamma/\ker\rho_A\), with the natural quotient maps
\(q_{BA}(\gamma\ker\rho_B)=\gamma\ker\rho_A\), also form an inverse system.

\begin{lemma}
\label{lem:isoofprof}
The inverse systems
\[
(\Gamma/\ker\rho_A,q_{BA})_{A\subseteq B}
\qquad\text{and}\qquad
(\Gamma_A,r_{BA})_{A\subseteq B}
\]
are isomorphic. Consequently,
\[
\varprojlim_{A\in\mathcal F_\Gamma(\Omega_\Gamma)}
\Gamma/\ker\rho_A
\cong_{\mathrm{top}}
\varprojlim_{A\in\mathcal F_\Gamma(\Omega_\Gamma)}
\Gamma_A .
\]
\end{lemma}

\begin{proof}
For each \(A\), the first isomorphism theorem gives an isomorphism

\noindent \(\overline\rho_A:\Gamma/\ker\rho_A\to\Gamma_A\), defined by
\(\overline\rho_A(\gamma\ker\rho_A)=\rho_A(\gamma)\). 

\noindent If \(A\subseteq B\),
then for every \(\gamma\in\Gamma\),
\[
r_{BA}\bigl(\overline\rho_B(\gamma\ker\rho_B)\bigr)
=
\rho_A(\gamma)
=
\overline\rho_A\bigl(q_{BA}(\gamma\ker\rho_B)\bigr).
\]
Thus the maps \(\overline\rho_A\) are compatible with the transition maps, and
therefore give an isomorphism of inverse systems. Since all groups involved are
finite and discrete, the induced isomorphism on inverse limits is a
topological isomorphism.
\end{proof}

\begin{lemma}
\label{lem:kernels-normal-quotient-action}
The family \((\ker\rho_A)_{A\in\mathcal F_\Gamma(\Omega_\Gamma)}\) coincides
with \(\mathcal N_o(\Gamma)\).
\end{lemma}

\begin{proof}
Let \(A=\Gamma/N_1\coprod\cdots\coprod\Gamma/N_r\). The left-translation
action of \(\Gamma\) on the component \(\Gamma/N_i\) has kernel \(N_i\).
Therefore the kernel of the restricted action on \(A\) is
$\ker\rho_A=N_1\cap\cdots\cap N_r$.
This is an open normal subgroup of \(\Gamma\).

Conversely, if \(N\in\mathcal N_o(\Gamma)\), take \(A=\Gamma/N\). The kernel
of the action of \(\Gamma\) on this component is exactly \(N\). Hence every
member of \(\mathcal N_o(\Gamma)\) occurs as some \(\ker\rho_A\).
\end{proof}

We now apply the finite-orbit construction to the action of \(\Gamma\) on
\(\Omega_\Gamma\), and write
\[
\Gamma_{\Omega_\Gamma}:=
\varprojlim_{A\in\mathcal F_\Gamma(\Omega_\Gamma)}\Gamma_A .
\]

\begin{theorem}
\label{thm:normal-quotient-action-recovers-gamma}
There is a canonical topological isomorphism
\[
\Gamma_{\Omega_\Gamma}\cong_{\mathrm{top}}\Gamma .
\]
\end{theorem}

\begin{proof}
By Lemma~\ref{lem:isoofprof},
\[
\Gamma_{\Omega_\Gamma}
\cong_{\mathrm{top}}
\varprojlim_{A\in\mathcal F_\Gamma(\Omega_\Gamma)}\Gamma/\ker\rho_A.
\]
By Lemma~\ref{lem:kernels-normal-quotient-action}, the subgroups
\(\ker\rho_A\) are precisely the open normal subgroups of \(\Gamma\). Hence
\[
\varprojlim_{A\in\mathcal F_\Gamma(\Omega_\Gamma)}\Gamma/\ker\rho_A
\cong_{\mathrm{top}}
\varprojlim_{N\in\mathcal N_o(\Gamma)}\Gamma/N.
\]
The latter inverse limit is canonically topologically isomorphic to
\(\Gamma\); see \cite[Proposition~1.3]{DDMS99}. Therefore
\(\Gamma_{\Omega_\Gamma}\cong_{\mathrm{top}}\Gamma\).
\end{proof}

Let
$\rho:\Gamma\longrightarrow\Sym(\Omega_\Gamma)
$ 
be the permutation representation associated with the normal finite-quotient
action.

\begin{corollary}
\label{cor:normal-action-closed-image}
The group \(\Gamma\) is topologically isomorphic to \(\rho(\Gamma)\). In
particular, \(\Gamma\) is topologically isomorphic to a closed subgroup of
\(\Sym(\Omega_\Gamma)\).
\end{corollary}

\begin{proof}
By Theorem~\ref{thm:normal-quotient-action-recovers-gamma},
\(\Gamma\cong_{\mathrm{top}}\Gamma_{\Omega_\Gamma}\). By
Theorem~\ref{thm:permutation-closure},
\(\Gamma_{\Omega_\Gamma}\cong_{\mathrm{top}}\overline{\rho(\Gamma)}\). Since
\(\Gamma\) is compact and \(\Sym(\Omega_\Gamma)\) is Hausdorff,
\(\rho(\Gamma)\) is compact and therefore closed. Thus
\(\overline{\rho(\Gamma)}=\rho(\Gamma)\), and the result follows.
\end{proof}

\subsection*{Exactness for the normal finite-quotient action}
\label{subsec:normal-quotient-action-exactness}

We have shown that the normal finite-quotient action reconstructs the
profinite group \(\Gamma\). We now determine when this action satisfies
\textup{FLEP}. Since the components \(\Gamma/N\) involve only open normal
subgroups \(N\), they do not directly detect arbitrary open subgroups. Thus exactness can hold only under an additional normality condition. We now
identify this condition.

Throughout this subsection, fixed points and stabilizers are taken with
respect to the normal finite-quotient action defined above. The following
lemma gives the basic computation needed for the \textup{FLEP} criterion.

\begin{lemma}
\label{lem:fix-normal-quotient-action}
For every subgroup \(U\leq\Gamma\),
\[
\Fix_{\Omega_\Gamma}(U)
=
\coprod_{\substack{N\in\mathcal N_o(\Gamma)\\ U\leq N}}\Gamma/N
\quad\text{and}\quad
\Stab_\Gamma(\Fix_{\Omega_\Gamma}(U))
=
\bigcap_{\substack{N\in\mathcal N_o(\Gamma)\\ U\leq N}}N.
\]
\end{lemma}

\begin{proof}
Let \(N\in\mathcal N_o(\Gamma)\) and \(\delta N\in\Gamma/N\). For \(u\in U\),
one has \(u(\delta N)=u\delta N\). Thus \(\delta N\) is fixed by every element
of \(U\) if and only if \(u\delta N=\delta N\) for all \(u\in U\), equivalently
\(\delta^{-1}u\delta\in N\) for all \(u\in U\). Since \(N\) is normal, this is
equivalent to \(U\leq N\).

Hence, if \(U\leq N\), the whole component \(\Gamma/N\) is fixed pointwise by
\(U\), while if \(U\nleq N\), no point of that component is fixed by \(U\). This
proves the formula for \(\Fix_{\Omega_\Gamma}(U)\).

An element \(\gamma\in\Gamma\) fixes every point of
\(\Fix_{\Omega_\Gamma}(U)\) if and only if it fixes each component 
\(\Gamma/N\) pointwise, for every \(N\in\mathcal N_o(\Gamma)\) with
\(U\leq N\). Hence \(\gamma\) lies in
\(\Stab_\Gamma(\Fix_{\Omega_\Gamma}(U))\) if and only if
\(\gamma\in N\) for every \(N\in\mathcal N_o(\Gamma)\) with \(U\leq N\).
This gives the second formula.
\end{proof}

Recall that a group is called a Dedekind group if every subgroup is normal;
see \cite{RobinsonGroups}.

\begin{theorem}
\label{thm:FLEP-normal-quotient-action}
Let \(\Gamma\) be a profinite group, and let
\(\Omega_\Gamma=\coprod_{N\in\mathcal N_o(\Gamma)}\Gamma/N\) be the normal
finite-quotient \(\Gamma\)-set, with the left-translation action. 

The following
conditions are equivalent:
\begin{enumerate}
\item[\textup{(i)}] The action of \(\Gamma\) on \(\Omega_\Gamma\) satisfies
\textup{FLEP}.
\item[\textup{(ii)}] Every open subgroup of \(\Gamma\) is normal.
\item[\textup{(iii)}] For every \(N\in\mathcal N_o(\Gamma)\), the finite group
\(\Gamma/N\) is a Dedekind group.
\end{enumerate}
\end{theorem}

\begin{proof}
\(\textup{(i)}\Rightarrow\textup{(ii)}\). Assume that the action satisfies
\textup{FLEP}, and let \(U\) be an open subgroup of \(\Gamma\). Then
\(U=\Stab_\Gamma(\Fix_{\Omega_\Gamma}(U))\). By
Lemma~\ref{lem:fix-normal-quotient-action},
\[
U=\bigcap_{\substack{N\in\mathcal N_o(\Gamma)\\ U\leq N}}N.
\]
The right-hand side is an intersection of normal subgroups of \(\Gamma\). 
 Therefore \(U\) is normal in \(\Gamma\). Hence every open subgroup of
\(\Gamma\) is normal.

\(\textup{(ii)}\Rightarrow\textup{(i)}\). Let \(U\leq\Gamma\) be open. By
assumption, \(U\in\mathcal N_o(\Gamma)\). Hence \(U\) itself occurs among the
open normal subgroups \(N\) with \(U\leq N\). Therefore
\[
\bigcap_{\substack{N\in\mathcal N_o(\Gamma)\\ U\leq N}}N=U.
\]
By Lemma~\ref{lem:fix-normal-quotient-action},
$\Stab_\Gamma(\Fix_{\Omega_\Gamma}(U))=U$.

Thus the action satisfies \textup{FLEP}.

\noindent \(\textup{(ii)}\Rightarrow\textup{(iii)}\). Let
\(N\in\mathcal N_o(\Gamma)\), and let \(Q\leq\Gamma/N\). Then \(Q=U/N\) for
some subgroup \(U\leq\Gamma\) containing \(N\). Since \(N\) is open, \(U\) is
open in \(\Gamma\). By assumption, \(U\) is a normal subgroup of \(\Gamma\), and therefore
\(Q=U/N\) is a normal subgroup of \(\Gamma/N\). Thus every subgroup of \(\Gamma/N\) is
normal, so \(\Gamma/N\) is a Dedekind group.

\(\textup{(iii)}\Rightarrow\textup{(ii)}\). Assume that every finite quotient
\(\Gamma/N\), with \(N\in\mathcal N_o(\Gamma)\), is a Dedekind group. Let
\(U\leq_o\Gamma\). Since the open normal subgroups form a neighbourhood basis
of the identity, there exists \(N\in\mathcal N_o(\Gamma)\) with \(N\leq U\).
Then \(U/N\leq\Gamma/N\), hence \(U/N\) is a normal subgroup of  \(\Gamma/N\). It follows
that \(U\) is a normal subgroup of  \(\Gamma\). Therefore every open subgroup of \(\Gamma\)
is normal.
\end{proof}
The preceding theorem shows that the normal finite-quotient action need not
satisfy \textup{FLEP}. Its exactness is determined precisely by the condition that all quotients
\(\Gamma/N\), with \(N\in\mathcal N_o(\Gamma)\), are Dedekind groups. This gives the
following class of profinite groups with an exact fixed-set correspondence.

\begin{corollary}[Galois-type correspondence for profinite Dedekind limits]
Let \((D_i,d_{ji})_{i\leq j}\) be an inverse system of finite Dedekind groups
with surjective transition maps, and let \(\Gamma:=\varprojlim_i D_i\). 

Then
the normal finite-quotient action of \(\Gamma\) on
\(\Omega_\Gamma=\coprod_{N\in\mathcal N_o(\Gamma)}\Gamma/N\) satisfies
\textup{FLEP}. Consequently, the assignments
\[
J\longmapsto \Fix_{\Omega_\Gamma}(J),
\qquad
X\longmapsto \Stab_\Gamma(X)
\]
define mutually inverse inclusion-reversing bijections between the closed
subgroups \(J\leq\Gamma\) and the fixed subsets
\(X=\Fix_{\Omega_\Gamma}(J)\) arising from closed subgroups.
\end{corollary}

\begin{proof}
Since the transition maps are surjective, every finite continuous quotient of
\(\Gamma\) is a quotient of some \(D_i\). Hence, for every
\(N\in\mathcal N_o(\Gamma)\), the finite group \(\Gamma/N\) is a quotient of
some \(D_i\). Since quotients of Dedekind groups are Dedekind, every
\(\Gamma/N\) is Dedekind. The result now follows from
Theorem~\ref{thm:FLEP-normal-quotient-action} and
Theorem~\ref{thm:fixed-set-correspondence}.
\end{proof}

Thus inverse limits of finite Dedekind groups provide a concrete class of
profinite groups for which reconstruction is accompanied by an exact
fixed-set correspondence.

\section{Application to infinite Galois theory}
\label{sec:galois}

We now apply the finite-orbit construction to field extensions. The
finite-orbit hypothesis used here is the field-theoretic analogue of the local
finiteness condition for automorphism groups of commutative rings studied by
Cruthirds~\cite{Cruthirds}.

Let \(E\) be a field and let \(G\leq\Aut(E)\). Thus \(G\) acts on \(E\) by
field automorphisms. In the notation of the preceding sections, the fixed set
of this action is
\[
\Fix_E(G)=\{x\in E:g(x)=x\text{ for all }g\in G\}.
\]
In field theory this fixed set is the fixed field of \(G\), usually denoted by
\(E^G\). We set \(F:=E^G=\Fix_E(G)\).

\begin{proposition}
\label{prop:FOP-galois-equivalence}
The action of \(G\) on \(E\) has \textup{FOP} if and only if \(E/F\) is
Galois.
\end{proposition}

\begin{proof}
Assume first that the action has \textup{FOP}. 

\noindent For \(x\in E\), the orbit
\(Gx=\{g(x):g\in G\}\) is finite. Hence the polynomial
\[
p_x(T):=\prod_{y\in Gx}(T-y)
\]
has coefficients fixed by \(G\), so \(p_x(T)\in F[T]\). Since its roots are
distinct and \(x\) is one of them, \(x\) is algebraic and separable over \(F\).
If \(m_x(T)\in F[T]\) is the minimal polynomial of \(x\), then
\(m_x(T)\) divides \(p_x(T)\), so all roots of \(m_x(T)\) are distinct and lie in \(Gx\subseteq
E\). Hence \(E/F\) is separable and normal, therefore Galois.

Conversely, assume that \(E/F\) is Galois. For \(x\in E\), every element
\(g(x)\) of its orbit \(Gx\) is an \(F\)-conjugate of \(x\). Since \(x\) has
only finitely many \(F\)-conjugates, \(Gx\) is finite. Hence the action has
\textup{FOP}.\end{proof}

From now on, assume that the equivalent conditions of
Proposition~\ref{prop:FOP-galois-equivalence} hold. Thus \(E/F\) is Galois.

Let \(\mathcal F_G(E)\) be the directed set of finite \(G\)-stable subsets of
\(E\). 

\noindent For \(A\in\mathcal F_G(E)\), set \(L_A:=F(A)\), the subfield of \(E\)
generated over \(F\) by the elements of~\(A\).

\begin{lemma}
\label{lem:finite-level-galois}
For every \(A\in\mathcal F_G(E)\), the extension \(L_A/F\) is finite Galois.
\end{lemma}

\begin{proof}
Write \(A=\{a_1,\ldots,a_n\}\). Since \(A\) is \(G\)-stable, the polynomial
\(p_A(T):=\prod_{i=1}^n(T-a_i)\) has coefficients in \(F\). The field
\(L_A=F(A)\) is the splitting field over \(F\) of the separable polynomial
\(p_A(T)\). Hence \(L_A/F\) is finite Galois.
\end{proof}

For \(A\in\mathcal F_G(E)\), let \(\rho_A:G\to\Sym(A)\) be the permutation
representation induced by the action of \(G\) on \(A\), with image
\(G_A:=\rho_A(G)\).

Given \(g\in\Aut(E)\) and a subfield \(K\subseteq E\), we denote by
\(g|_{K}\) the restriction of \(g\) to \(K\).

\begin{lemma}
\label{lem:finite-level-restrictions}
For every \(A\in\mathcal F_G(E)\), the set
\(\{g|_{L_A}:g\in G\}\) is a subgroup of \(\Gal(L_A/F)\), and equals
\(\Gal(L_A/F)\).
\end{lemma}

\begin{proof}
Let $H_A:=\{g|_{L_A}:g\in G\}$.
Since \(A\) is \(G\)-stable and \(L_A=F(A)\), every \(g\in G\) restricts to an
\(F\)-automorphism of \(L_A\). Thus
$H_A\leq\Gal(L_A/F)$.

We show that the fixed field of \(H_A\) in \(L_A\) is \(F\). Indeed,
\[
L_A^{H_A}
=
\{x\in L_A:g(x)=x\text{ for every }g\in G\}
=
L_A\cap E^G
=
F.
\]
By Artin's theorem applied to the finite group \(H_A\leq\Aut(L_A)\), the
extension \(L_A/L_A^{H_A}\) is Galois and
$\Gal(L_A/L_A^{H_A})=H_A$.

Since \(L_A^{H_A}=F\), this gives
$H_A=\Gal(L_A/F)$.
\end{proof}

\begin{lemma}
\label{lem:theta-A-galois}
For every \(A\in\mathcal F_G(E)\), there is a canonical group isomorphism
\[
\theta_A:G_A\longrightarrow\Gal(L_A/F),
\qquad
\theta_A(\rho_A(g))=g|_{L_A}.
\]
\end{lemma}

\begin{proof}
If \(\rho_A(g)=\rho_A(h)\), then \(g\) and \(h\) agree on \(A\). Since both
fix \(F\) pointwise and \(L_A=F(A)\), they agree on \(L_A\). Hence
\(\theta_A\) is well-defined. It is plainly a homomorphism. It is injective
because equality of the restrictions to \(L_A\) implies equality on \(A\), and
therefore equality in \(G_A\). It is surjective by
Lemma~\ref{lem:finite-level-restrictions}.
\end{proof}

For \(A\subseteq B\), we have \(L_A\subseteq L_B\). We denote by
\(r_{BA}:G_B\to G_A\) the restriction homomorphism
\(r_{BA}(\rho_B(g))=\rho_A(g)\). 

\noindent We also denote by
\(\operatorname{res}_{L_B/L_A}:\Gal(L_B/F)\to\Gal(L_A/F)\),
\(\sigma\mapsto\sigma|_{L_A}\), the usual restriction homomorphism.
\begin{lemma}
\label{lem:theta-compatible}
Let \(A,B\in\mathcal F_G(E)\) with \(A\subseteq B\). Then the diagram
\[
\begin{tikzcd}
G_B \arrow[r, "\theta_B"] \arrow[d, "r_{BA}"']
&
\Gal(L_B/F) \arrow[d, "\operatorname{res}_{L_B/L_A}"]
\\
G_A \arrow[r, "\theta_A"']
&
\Gal(L_A/F)
\end{tikzcd}
\]
commutes.
\end{lemma}

\begin{proof}
For \(g\in G\),
$\operatorname{res}_{L_B/L_A}\bigl(\theta_B(\rho_B(g))\bigr)
=
g|_{L_A}
=
\theta_A\bigl(r_{BA}(\rho_B(g))\bigr)$.

Hence the diagram commutes.
\end{proof}
\noindent We now introduce the two profinite groups to be compared. 

The finite-orbit
construction gives
$\Gamma_E:=\varprojlim_{A\in\mathcal F_G(E)}G_A$.

The finite Galois extensions \(L_A/F\) give
$\Lambda_E:=\varprojlim_{A\in\mathcal F_G(E)}\Gal(L_A/F)$,
where the transition maps are the restrictions
\(\operatorname{res}_{L_B/L_A}\).

\begin{proposition}
\label{prop:gamma-lambda-E}
There is a canonical topological isomorphism
\[
\Gamma_E\cong_{\mathrm{top}}\Lambda_E.
\]
\end{proposition}

\begin{proof}
By Lemma~\ref{lem:theta-compatible}, the isomorphisms
\(\theta_A:G_A\to\Gal(L_A/F)\) form an isomorphism of inverse systems
\[
(G_A,r_{BA})_{A\subseteq B}
\qquad\text{and}\qquad
(\Gal(L_A/F),\operatorname{res}_{L_B/L_A})_{A\subseteq B}.
\]
They therefore induce an isomorphism \(\Gamma_E\to\Lambda_E\). Since all
groups in the inverse systems are finite and discrete, the induced isomorphism
is a homeomorphism.
\end{proof}

Since the action has \textup{FOP}, for every \(x\in E\) the orbit \(Gx\) is
finite and belongs to \(\mathcal F_G(E)\). Since \(x\in F(Gx)=L_{Gx}\), we have
\[
E=\bigcup_{A\in\mathcal F_G(E)}L_A.
\]
The group \(\Lambda_E\) acts naturally on \(E\). Indeed, if
\(\lambda=(\lambda_A)_A\in\Lambda_E\) and \(x\in E\), choose
\(A\in\mathcal F_G(E)\) such that \(x\in L_A\), and define
\(\lambda(x):=\lambda_A(x)\). This is independent of the choice of \(A\), by
compatibility. Compatibility also shows that this defines a left action of
\(\Lambda_E\) on \(E\).

\begin{proposition}
\label{prop:lambda-E-FLEP}
The induced action of \(\Lambda_E\) on \(E\) satisfies \textup{FLEP}.
\end{proposition}
\begin{proof}
We verify condition \(\textup{(F2)}\) of
Proposition~\ref{prop:FLEP-equivalent-forms}. Let \(U\leq\Lambda_E\) be open
and let \(\lambda\in\Lambda_E\setminus U\). By the description of open
subgroups in an inverse limit of finite groups, there exist
\(A\in\mathcal F_G(E)\) and a subgroup \(Q\leq\Gal(L_A/F)\) such that
$U=\pi_A^{-1}(Q)$.

Since \(\lambda\notin U\), we have \(\lambda_A\notin Q\).
By finite Galois theory,
$Q=\Gal(L_A/L_A^Q)$.
Hence \(\lambda_A\notin\Gal(L_A/L_A^Q)\), so there exists
\(x\in L_A^Q\) such that
$\lambda_A(x)\neq x$.

We claim that \(x\in\Fix_E(U)\). Let \(u=(u_B)_B\in U\). Since
\(U=\pi_A^{-1}(Q)\), one has \(u_A\in Q\). Since \(x\in L_A^Q\), it follows
that \(u_A(x)=x\). By the definition of the action of \(\Lambda_E\) on \(E\),
this gives \(u(x)=x\). Hence \(x\in\Fix_E(U)\).

On the other hand, since \(x\in L_A\), the action of \(\lambda\) on \(x\) is
given by the \(A\)-coordinate \(\lambda_A\). Therefore
$\lambda(x)=\lambda_A(x)\neq x$.

Thus there exists \(x\in\Fix_E(U)\) such that \(\lambda(x)\neq x\). This is
condition \(\textup{(F2)}\), and therefore the induced action of
\(\Lambda_E\) on \(E\) satisfies \textup{FLEP}.
\end{proof}

The natural topology on \(\Gal(E/F)\) for comparison with \(\Gamma_E\) is the
Krull topology. Indeed, via the canonical isomorphisms
\(G_A\simeq\Gal(L_A/F)\), the finite permutation levels defining
\(\Gamma_E\) are identified with the Galois groups of the finite Galois
subextensions \(L_A/F\). Hence the natural comparison with \(\Gal(E/F)\) is
given by the restriction homomorphisms
\[
\Gal(E/F)\longrightarrow \Gal(L_A/F),
\qquad
\sigma\longmapsto \sigma|_{L_A}.
\]
The topology induced by these finite-level restriction maps is precisely the
Krull topology. Indeed, if \(L/F\) is a finite Galois subextension of \(E/F\),
then \(L\) is the splitting field over \(F\) of a separable polynomial
\(f(T)\in F[T]\). If \(A\) is the set of roots of \(f\) in \(E\), then
\(A\in\mathcal F_G(E)\) and \(L=L_A\). Equivalently, a neighbourhood basis of
the identity is given by the subgroups \(\Gal(E/L)\), where \(L/F\) ranges
over the finite Galois subextensions of \(E/F\); see, for example,
\cite{ConradInfiniteGalois, LangAlgebra}.
\begin{theorem}
\label{thm:GammaE-GalEF}
There is a canonical topological isomorphism
\[
\Gamma_E\cong_{\mathrm{top}}\Gal(E/F),
\]
where \(\Gal(E/F)\) is endowed with the Krull topology.
\end{theorem}

\begin{proof}
By Proposition~\ref{prop:gamma-lambda-E}, it is enough to identify
\(\Lambda_E\) with \(\Gal(E/F)\). As observed above,
\[
E=\bigcup_{A\in\mathcal F_G(E)}L_A.
\]

Since each \(L_A/F\) is a finite Galois subextension of \(E/F\), every
\(F\)-auto\-morphism of \(E\) stabilizes \(L_A\). Hence restriction defines a homomorphism
\[
\Psi:\Gal(E/F)\longrightarrow\Lambda_E,
\qquad
\Psi(\sigma)=(\sigma|_{L_A})_A.
\]
It is injective, since an automorphism of \(E\) is determined by its
restrictions to the fields \(L_A\), which cover \(E\).

We prove that \(\Psi\) is surjective. Let
\((\sigma_A)_A\in\Lambda_E\). For \(x\in E\), choose
\(A\in\mathcal F_G(E)\) such that \(x\in L_A\), and define
\(\sigma(x):=\sigma_A(x)\).

This is independent of the choice of \(A\). Indeed, if also
\(x\in L_B\), then \(C:=A\cup B\) belongs to \(\mathcal F_G(E)\), and the
restrictions of \(\sigma_C\) to \(L_A\) and \(L_B\) are \(\sigma_A\) and
\(\sigma_B\), respectively. Hence
\(\sigma_A(x)=\sigma_C(x)=\sigma_B(x)\).

The map \(\sigma\) is an \(F\)-automorphism of \(E\). To see this, let
\(x,y\in E\). Then \(Gx,Gy\in\mathcal F_G(E)\), and for
\(A:=Gx\cup Gy\) one has \(x,y\in L_A\). Additivity, multiplicativity, and
preservation of \(1_E\) follow from the fact that \(\sigma_A\) is a field
automorphism of \(L_A\). Also \(\sigma\) fixes \(F\), since each \(\sigma_A\)
is an \(F\)-automorphism. The family \((\sigma_A^{-1})_A\) is again compatible.  Applying the same
construction to it gives the inverse map. Hence
\(\sigma\in\Gal(E/F)\), and by construction
\(\Psi(\sigma)=(\sigma_A)_A\). Thus \(\Psi\) is surjective.
Therefore \(\Psi\) is a group isomorphism.

The map \(\Psi\) is continuous for the Krull topology. 

\noindent 
Indeed, for each
\(A\in\mathcal F_G(E)\), the inverse image under \(\Psi\) of the kernel of the
projection \(\Lambda_E\to\Gal(L_A/F)\) is \(\Gal(E/L_A)\), which is open in
the Krull topology. Since \(\Psi\) is a continuous bijection from the compact
group \(\Gal(E/F)\) onto the Hausdorff group \(\Lambda_E\), it is a
homeomorphism.

Hence \(\Gal(E/F)\cong_{\mathrm{top}}\Lambda_E\), and the result follows from
Proposition~\ref{prop:gamma-lambda-E}.
\end{proof}

\section{Concluding remarks}

The preceding results show that finite-orbit actions provide a natural setting
for profinite reconstruction. The construction recovers the closure of the
original permutation representation, while the additional condition
\textup{FLEP} determines when this reconstruction carries an exact fixed-point
correspondence.

The profinite and Galois applications illustrate two complementary aspects of
the theory. In the profinite case, the normal finite-quotient action always
reconstructs the group, but it is exact only when every open subgroup is
normal, equivalently when every finite quotient is a Dedekind group. In the
field-theoretic case, finite orbits are equivalent to the Galois property of
\(E/E^G\), and the associated profinite group is the Krull Galois group.

Thus the finite-orbit viewpoint separates two phenomena: reconstruction of a
profinite group from finite permutation data, and exact recovery of subgroups
from fixed-point sets. The main role of \textup{FLEP} is to identify precisely
when these two phenomena occur together.


\begin{thebibliography}{99}


\bibitem{ConradInfiniteGalois}
Conrad, K.:
Infinite Galois Theory, draft lecture notes, CTNT 2020.
Available at
\url{https://ctnt-summer.math.uconn.edu/wp-content/uploads/sites/1632/2020/06/CTNT-InfGaloisTheory.pdf}.

\bibitem{Cruthirds}
Cruthirds, J.~E.:
\mbox{Infinite Galois theory for commutative rings.}
\\
\mbox{Pacific~J.~Math.~\textbf{64} (1976), no.~1, 107--117.}

\bibitem{DDMS99}
Dixon, J.~D., du Sautoy, M.~P.~F., Mann, A., Segal, D.:
Analytic Pro-\(p\) Groups.
2nd ed., Cambridge Studies in Advanced Mathematics, vol.~61.
Cambridge University Press, Cambridge, 1999.

\bibitem{LangAlgebra}
Lang, S.:
Algebra.
Graduate Texts in Mathematics, vol.~211, 3rd ed.
Springer, New York, 2002.

\bibitem{ReidWesolek}
Reid, C.~D., Wesolek, P.~R.:
Homomorphisms into totally disconnected, locally compact groups with dense image.
Forum Math. \textbf{31} (2019), no.~3, 685--701.

\bibitem{RibesZalesskii}
Ribes, L., Zalesskii, P.:
Profinite Groups.
2nd ed., Ergebnisse der Mathematik und ihrer Grenzgebiete, vol.~40.
Springer, Berlin, 2010.

\bibitem{RobinsonGroups}
Robinson, D.~J.~S.:
\emph{A Course in the Theory of Groups}.
2nd ed., Graduate Texts in Mathematics, vol.~80.
Springer, New York, 1996.

\bibitem{Wilson}
Wilson, J.~S.:
Profinite Groups.
London Mathematical Society Monographs, New Series, vol.~19.
Clarendon Press, Oxford, 1998.
\end{thebibliography}
\end{document}